\documentclass[a4paper, reqno, 12pt]{amsart}

\usepackage[usenames,dvipsnames]{color}
\usepackage{amsthm,amsfonts,amssymb,amsmath,amsxtra}
\usepackage[all]{xy}
\SelectTips{cm}{}
\usepackage{xr-hyper}
\usepackage[colorlinks=true, 
   citecolor=Black,
   linkcolor=Red,
   urlcolor=Blue]{hyperref}
\usepackage{verbatim}

\usepackage[margin=1.25in]{geometry}
\usepackage{mathrsfs}
\usepackage{stmaryrd}

\def\hmath$#1${\texorpdfstring{{\rmfamily\textit{#1}}}{#1}}

\RequirePackage{xspace}
\RequirePackage{etoolbox}
\RequirePackage{varwidth}
\RequirePackage{enumitem}
\RequirePackage{tensor}
\RequirePackage{mathtools}
\RequirePackage{longtable}
\RequirePackage{multirow}

\def\ge{\geqslant}
\def\le{\leqslant}
\def\a{\alpha}
\def\b{\beta}

\def\G{\Gamma}
\def\d{\delta}

\def\s{\sigma}
\def\t{\tau}

\def\k{\kappa}
\def\l{\lambda}

\def\i{^{-1}}

\def\<{\langle}
\def\>{\rangle}

\newcommand{\fka}{\ensuremath{\mathfrak{a}}\xspace}

\newcommand{\bG}{\mathbf G}

\newcommand{{\BG}}{\ensuremath{\mathbb {G}}\xspace}

\newcommand{{\BK}}{\ensuremath{\mathbb {K}}\xspace}

\newcommand{\BR}{\ensuremath{\mathbb {R}}\xspace}
\newcommand{\BS}{\ensuremath{\mathbb {S}}\xspace}

\newcommand{\BZ}{\ensuremath{\mathbb {Z}}\xspace}

\newcommand{\CA}{\ensuremath{\mathcal {A}}\xspace}

\newcommand{\Ad}{{\mathrm{Ad}}}
\newcommand{\ad}{{\mathrm{ad}}}

\DeclareMathOperator{\Gal}{Gal}

\def\tW{\tilde W}

\def\wtt{{\rm wt}}

\newtheorem{theorem}{Theorem}
\newtheorem{proposition}[theorem]{Proposition}
\newtheorem{lemma}[theorem]{Lemma}

\newtheorem{corollary}[theorem]{Corollary}

\theoremstyle{definition}

\newtheorem{remark}[theorem]{Remark}

\numberwithin{equation}{section}
\numberwithin{theorem}{section}

\setitemize[0]{leftmargin=*,itemsep=\the\smallskipamount}
\setenumerate[0]{leftmargin=*,itemsep=\the\smallskipamount}

\renewcommand{\to}{%
   \ifbool{@display}{\longrightarrow}{\rightarrow}%
   }
\let\shortmapsto\mapsto
\renewcommand{\mapsto}{%
   \ifbool{@display}{\longmapsto}{\shortmapsto}%
   }
\newlength{\olen}
\newlength{\ulen}
\newlength{\xlen}
\newcommand{\xra}[2][]{%
   \ifbool{@display}%
      {\settowidth{\olen}{$\overset{#2}{\longrightarrow}$}%
       \settowidth{\ulen}{$\underset{#1}{\longrightarrow}$}%
       \settowidth{\xlen}{$\xrightarrow[#1]{#2}$}%
       \ifdimgreater{\olen}{\xlen}%
          {\underset{#1}{\overset{#2}{\longrightarrow}}}%
          {\ifdimgreater{\ulen}{\xlen}%
             {\underset{#1}{\overset{#2}{\longrightarrow}}}
             {\xrightarrow[#1]{#2}}}}%
      {\xrightarrow[#1]{#2}}
   }
\makeatother
\newcommand{\xyra}[2][]{%
   \settowidth{\xlen}{$\xrightarrow[#1]{#2}$}%
   \ifbool{@display}%
      {\settowidth{\olen}{$\overset{#2}{\longrightarrow}$}%
       \settowidth{\ulen}{$\underset{#1}{\longrightarrow}$}%
       \ifdimgreater{\olen}{\xlen}%
          {\mathrel{\xymatrix@M=.12ex@C=3.2ex{\ar[r]^-{#2}_-{#1} &}}}%
          {\ifdimgreater{\ulen}{\xlen}%
             {\mathrel{\xymatrix@M=.12ex@C=3.2ex{\ar[r]^-{#2}_-{#1} &}}}
             {\mathrel{\xymatrix@M=.12ex@C=\the\xlen{\ar[r]^-{#2}_-{#1} &}}}}}%
      {\mathrel{\xymatrix@M=.12ex@C=\the\xlen{\ar[r]^-{#2}_-{#1} &}}}%
   }
\makeatletter
\newcommand{\xla}[2][]{%
   \ifbool{@display}%
      {\settowidth{\olen}{$\overset{#2}{\longleftarrow}$}%
       \settowidth{\ulen}{$\underset{#1}{\longleftarrow}$}%
       \settowidth{\xlen}{$\xleftarrow[#1]{#2}$}%
       \ifdimgreater{\olen}{\xlen}%
          {\underset{#1}{\overset{#2}{\longleftarrow}}}%
          {\ifdimgreater{\ulen}{\xlen}%
             {\underset{#1}{\overset{#2}{\longleftarrow}}}
             {\xleftarrow[#1]{#2}}}}%
      {\xleftarrow[#1]{#2}}
   }
\newcommand{\isoarrow}{%
   \ifbool{@display}{\overset{\sim}{\longrightarrow}}{\xrightarrow\sim}%
   }
   
\begin{document}

\title[]{Demazure product of the affine Weyl groups}
\author{Xuhua He}
\address{X.~H., The Institute of Mathematical Sciences and Department of Mathematics, The Chinese University of Hong Kong, Shatin, N.T., Hong Kong SAR, China}
\email{xuhuahe@math.cuhk.edu.hk}
\author{Sian Nie}
\address{S.~N., University of Chinese Academy of Sciences \&  Academy of Mathematics and Systems Science, CAS, 100190, Beijing, China}
\email{niesian@amss.ac.cn}



\begin{abstract}
The Demazure product gives a natural monoid structure on any Coxeter group. Such structure occurs naturally in many different areas in Lie Theory. This paper studies the Demazure product of an extended affine Weyl group $\tW$. The main discovery is a close connection between the Demazure product of an extended affine Weyl group and the quantum Bruhat graph of the finite Weyl group. As applications, we obtain explicit formulas on the generic Newton points and the Demazure products of elements in the lowest two-sided cell/shrunken Weyl chambers of $\tW$, and obtain an explicit formula on the Lusztig-Vogan map from the coweight lattice to the set of dominant coweights.
\end{abstract}

\maketitle

\section*{Introduction}
An extended affine Weyl group $\tW$ parameterizes the affine Schubert varieties in an affine flag variety. It plays an important role in the study of affine Hecke algebras and $p$-adic groups. The group $\tW$ has two descriptions: 
\begin{enumerate}
    \item it is a semidirect product of the coweight lattice with a finite Weyl group;
    
    \item it is a quasi-Coxeter group, i.e., $\tW=W_a \rtimes \Omega$, where $W_a$ is a Coxeter group and $\Omega$ is a group of length-preserving automorphisms on $W_a$. 
\end{enumerate}

The multiplication on $\tW$ can be described easily via Description (1). Description (1) is also very useful in the study of the center of $\tW$ and its affine Hecke algebras. 

Description (2) allows us to define the length function and Bruhat order on $\tW$. It is essential in the definition of affine Hecke algebras and is important in the study of the convolution product on the affine flag varieties. 

Moreover, the quasi-Coxeter structure allows us to define the Demazure product $\ast$ on $\tW$. The Demazure product gives a monoid structure on $\tW$, which is useful in the study of affine $0$-Hecke algebras and in Lusztig's theory of total positivity. 

It is a natural question to understand the Demazure product $\ast$ via description (1) of $\tW$. As we will see later in the introduction, the answer to this question has several interesting applications in Lie theory. 

Note that any element in $\tW$ is of the form $t^\l x$ for $\l \in X_*$ and $x \in W_0$. If $\l, \l'$ are both dominant, then $t^\l \ast t^{\l'}=t^{\l+\l'}$. However, for arbitrary coweights $\l$ and $\l'$, $t^\l \ast t^{\l'}$ does not equal $t^{\l+\l'}$. Some correction term is needed. 

The main result of this paper is that the translation part of the correction term is given by the quantum Bruhat graph introduced by Brenti, Fomin, and Postnikov in 1999. See Theorem \ref{main} for a precise statement. 

First, let us explain why the quantum Bruhat graph and the Demazure product are related. Lam and Shimozono discovered in \cite{LS} that the Bruhat orders of the elements in $\tW$ with very regular translation part could be described via the quantum Bruhat graph. Note that $w \ast w'$, by definition, is the maximal element in the set $\{x' y'; x' \le x, y' \le y\}$. It indicates a possible connection between the Demazure product and the quantum Bruhat graph. New evidences come from the recent study of generic Newton points. 

The generic Newton point of an element $w \in \tW$, roughly speaking, is the maximal Frobenius-twisted conjugacy class in the loop group that intersects the Bruhat cell corresponding to $w$. Milic\'{e}vi\'{c} in \cite{M} established a formula of the generic Newton point of elements in $\tW$ with ``superregular'' translation part via the quantum Bruhat graph. Finally, the first author in a very recent work \cite{He-Demazure} gave a formula of the generic Newton point for any element in $w$ in terms of the Demazure powers of $w$.

In the proof of Theorem \ref{main}, we do not use the explicit formula of in \cite{He-Demazure}, but the idea of using Demazure product to calculate the generic Newton point. We then compare it with the formula in \cite{M}. It is worth pointing out that although the proof of Theorem \ref{main} here is only half a page, it is not an easy result. Our proof uses as a black box the main result of \cite{M}, the proof of which involves a long and technical analysis on the paths in the quantum Bruhat graph encoding saturated chains in the strong Bruhat order on the affine Weyl group. 

Note that the Demazure product of $w$ with a translation element ``in the same direction'' as $w$ is easy to calculate. Thus one may artificially add a suitable superregular translation to $w$ (so that one may use \cite{M}), and then subtract this translation element in the end. This is how we bypass the ``superregularity'' assumption of \cite{M} in Theorem \ref{main}. The proof of Theorem \ref{main} is given in \S\ref{sec:2}. 

Now we discuss some applications of our main result. 

In \cite{M}, the formula for the generic Newton point of $w$ is given under two assumptions: the group is split, and the translation part of the element $w$ is superregular. Using the Demazure product, both assumptions can be much weakened. In proposition \ref{prop:newton}, we obtain the formula for any quasi-split adjoint groups and for any $w$ in the lowest two-sided cell (i.e. the shrunken Weyl chambers) of $\tW$. Note that the quasi-split adjoint assumption is not essential. In \S\ref{sec:3.2}, we explain how the generic Newton points for arbitrary reductive groups can be deduced from quasi-split adjoint groups. 

Note that our main result only determines the translation part of the Demazure product. However, for elements in the lowest two-sided cell (i.e. the shrunken Weyl chambers) of $\tW$, we have an explicit formula for the whole Demazure product. This is proposition \ref{prop:ww}. 

Finally, the Demazure product on $\tW$ induces a map from the pairs of coweights to the dominant coweights (see \S\ref{sec:3.4}). We give an explicit formula for this map in proposition \ref{prop:ast-l}. As a consequence, we obtain in corollary \ref{cor:LV} an explicit formula for the map \cite{LV} from coweights to dominant coweights. This answers a question of Lusztig and Vogan. 

\subsection*{Acknowledgments:} X.H.~is partially supported by the Hong Kong RGC grant 14300021. We thank Felix Schremmer for his comments on a previous version of this work. 

\section{Preliminary}

\subsection{Extended affine Weyl groups}\label{root}
Let $\mathfrak R=(X^*, R, X_*, R^\vee, \Pi)$ be a based reduced root datum, where $X^*$ is the weight lattice, $X_*$ is the coweight lattice, $R \subseteq X^*$ is the set of roots, $R^\vee \subseteq X_*$ is the set of coroots and $\Pi \subseteq R$ is the set of simple roots. Let $V=X_* \otimes \BR$. For any $\a \in R$, we have a reflection $s_\a$ on $V$ sending $v$ to $v-\langle \a, v\rangle \a^\vee$, where $\langle ~ , ~ \rangle : X^* \times X_* \rightarrow \BZ$ is the natural perfect pairing between $X^*$ and $X_*$. The reflections $s_\a$ generate the {\it finite Weyl group} $W_0=W_0(\mathfrak R)$ of $\mathfrak R$. Let $\BS_0=\{s_\a; \a \in \Pi\}$ be the set of simple reflections. Then $(W_0, \BS_0)$ is a Coxeter system. Set \begin{gather*}
W_a=W_a(\mathfrak R):=\BZ R^\vee \rtimes W_0=\{t^\l w; \l \in \BZ R^\vee, w \in W_0\}, \\
\tW=\tW(\mathfrak R):=X_* \rtimes W_0=\{t^\l w; \l \in X_*, w \in W_0\}.
\end{gather*}
We call $W_a$ the {\it affine Weyl group} and $\tW$ the {\it extended affine Weyl group}. 

Let $R^+ \subseteq R$ be the set of positive roots determined by $\Pi$. The {\it base alcove} is the set $\mathfrak a=\{v \in V; 0 < \< \a, v\> < 1 \text{ for every } \a \in R^+\}.$ Let $\tilde \BS$ be the set of reflections along the walls of the base alcove $\mathfrak a$. Then $\BS_0 \subseteq \tilde \BS$ and $(W_a, \tilde \BS)$ is a Coxeter group.

Let $\Omega$ be the isotropy group in $\tW$ of the base alcove $\mathfrak a$. Then $\tW=W_a \rtimes \Omega$. We extend the length function $\ell$ and the Bruhat order $\le$ on the Coxeter group $W_a$ to $\tW$ in a natural way. By \cite{IM}, we have $$\ell(t^\l w) = \sum_{\a>0, w\i(\a)>0} |\<\l, \a\>| + \sum_{\a>0, w\i(\a)<0} |\<\l, \a\> - 1|.$$

For any $J \subseteq \tilde S$, we denote by $W_J \subseteq W_a$ the parabolic subgroup generated by $J$. Let ${}^J \tW = \{w \in \tW; w = \min (W_J w)\}$, where the minimum is taken with respect to the Bruhat order $\le$. Set $\tW^J = ({}^J \tW)\i$. If moreover, $W_J$ is finite, then we denote by $w_J$ the longest element in $W_J$. We simply write $w_0$ for $w_{\BS_0}$. We set 
\begin{gather*} {}^J \tW_{\max}=\{w \in \tW; w=\max(W_J w)\}=\{w_J w; w \in {}^J \tW\}, \\ \tW^J_{\max}=({}^J \tW_{\max}) \i=\{w \in \tW; w=\max(w W_J)\}=\{w w_J; w \in \tW^J\}.\end{gather*} If $J \subset \BS_0$, then we write ${}^J W_0$ for ${}^J \tW \cap W_0$ and $W_0^J$ for $\tW^J \cap W_0$. 

Let $X_*^+ = \{\l \in X_*; \<\l, \a\> \ge 0 \text{ for all } \a \in \Pi\}$ be the set of dominant coweights. For any $\l \in X_*^+$, we set $I(\l)=\{s \in \BS_0; s(\l)=\l\}$. Each element of $\tW$ can be written uniquely as $x t^\mu y$ with $\mu \in X_*^+$ and $x, y \in W_0$ with $t^\mu y \in {}^{\BS_0} \tW$. Here the condition $t^\mu y \in {}^{\BS_0} \tW$ is equivalent to that $y \in {}^{I(\mu)} W_0$. In this case, $\ell(x t^\mu y)=\ell(x)+\ell(t^{\mu})-\ell(y)$.

\subsection{Demazure product} \label{product}
We follow \cite[\S 1]{He07}. For any $x, y \in \tW$, the subset $\{x y'; y' \le y\}$ (resp. $\{x' y; x' \le x\}$, $\{x' y'; x' \le x, y' \le y\}$) of $\tW$ contains a unique maximal element (with respect to the Bruhat order $\le$). Moreover, we have $$\max\{x y'; y' \le y\}=\max\{x' y; x' \le x\}=\max\{x' y'; x' \le x, y' \le y\}.$$ We denote this element by $x \ast y$ and we call it the {\it Demazure product} of $x$ and $y$. Moreover, $(\tW, \ast)$ is a monoid and the Demazure product can be determined by the following two rules
\begin{itemize}
    \item $x \ast y = x y$ if $x, y \in \tW$ such that $\ell(x y) = \ell(x) + \ell(y)$;
    \item $s \ast w = w$ if $s \in \tilde \BS$, $w \in \tW$ such that $s w < w$.
\end{itemize}

The Demazure product on $\tW$ is encoded in the corresponding 0-Hecke algebra $H_0$.  By definition, $H_0$ is a $\BZ$-algebra with a linear basis $\{t_w; w \in W\}$, subject to the following relations 
\begin{enumerate}
    \item $t_w t_{w'} = t_{w w'}$ for $w, w' \in \tW$ such that $\ell(w w') = \ell(w) + \ell(w')$;
    \item $t_s^2 = - t_s$ for $s \in \tilde \BS$.
\end{enumerate}
Then for $w, w' \in \tW$ we have $t_w t_{w'} = (-1)^{\ell(w) + \ell(w') -\ell(w \ast w')} t_{w \ast w'}$.

The Demazure product also occurs naturally in the convolution product of Bruhat cells, which we will discuss in \S\ref{sec:conv}. It has also found applications in Lusztig's theory of total positivity \cite{Lu-1} and \cite{Lu-2}, and more recently, in the study of generic Newton points \cite{He-Demazure}. 

\subsection{Quantum Bruhat graphs}
The quantum Bruhat graph was introduced by Brenti, Fomin, and Postnikov in \cite{BFP}. It is an extension of the graph formed by covering relations in the Bruhat order of a finite Weyl group. It is naturally related to the quantum cohomology rings of flag varieties, see \cite{FGP}. It has also found applications in the description of the Bruhat order of affine Weyl groups \cite{LS} and in the study of generic Newton points \cite{M}. 

By definition, the quantum Bruhat graph $\Gamma$ associated to the root datum $\mathfrak R$ is a directed graph with 
\begin{enumerate}
    \item vertices indexed by elements of $W_0$;
    \item upward arrows $w \rightharpoonup w s_\a$ with $\a \in R^+$ if $\ell(w s_\a) = \ell(w) + 1$;
    \item downward arrow $w \rightharpoondown w s_\a$ with $\a \in R^+$ if $\ell(w s_\a) = \ell(w) - \<2 \rho, \a^\vee\> + 1$.
\end{enumerate}
Here $\rho$ is half sum of the positive roots in $R$. The weight of an upward arrow is defined to be zero, and the weight of a downward arrow $w \rightharpoondown w'$ is defined to be $\a^\vee$. The weight of a directed path is the sum of the weights of edges in the path. Following \cite[Lemma 1]{P}, for $w, w' \in W_0$, $\wtt(w, w')$ is defined to be the weight of any/some shortest directed path in $\Gamma$ from $w$ to $w'$. Note that $\wtt(w, w')$ is a nonnegative integral combination of simple coroots in $\mathfrak R$. 

\smallskip

The main purpose of this paper is to establish an interesting connection between the Demazure product on $\tW$ and the quantum Bruhat graph of $\mathfrak R$. 

\begin{theorem}\label{main}
Let $\mu_1, \mu_2 \in X_*^+$ and $x, y \in W_0$. If $t^{\mu_1} y \in {}^{\BS_0} \tW$, $x t^{\mu_2} \in \tW^{\BS_0}_{\max}$, then 
\begin{enumerate}
    \item $\mu_1 + \mu_2 - \wtt(y\i, x) \in X_*^+$;
    
    \item $(t^{\mu_1} y) \ast (x t^{\mu_2}) \in W_0 t^{\mu_1+\mu_2 - \wtt(y\i, x)}$.
\end{enumerate}
\end{theorem}

\section{Proof of the Main theorem}\label{sec:2}
Theorem \ref{main} is a combinatorial statement, and it is possible to be proved in a purely combinatorial way. However, we will argue in a more geometric way, by relating the Demazure product to the generic Newton points and then using the work of Milic\'{e}vi\'{c} \cite{M} on the connection between the generic Newton points and the quantum Bruhat graphs. 

\subsection{Iwahori-Weyl groups}\label{sec:twbg}
First, we recall the Iwahori-Weyl groups and their relations with the extended affine Weyl groups in \S\ref{root}. 

Let $F$ be a non-archimedean field and let $\breve F$ be the completion of the maximal unramified extension of $F$. The Frobenius automorphism of $\breve F$ over $F$ is denoted by $\s$.

Let $\bG$ be a connected reductive group over $F$. Let $S$ be a maximal $\breve F$-split torus defined over $F$. Let $\CA$ be the apartment of $G_{\breve F}$ corresponding to $S_{\breve F}$. Fix a $\s$-stable alcove $\fka$ in $\CA$. Set $\breve G = \bG(\breve F)$ and let $\breve I \subseteq \breve G$ be the Iwahori subgroup corresponding to $\fka$. Then $\breve I$ is $\s$-stable. Let $T$ be the centralizer of $S$ in $\bG$ and $N$ be the normalizer of $T$ in $\bG$. The Iwahori-Weyl group is defined by $$\tW(\bG) = N(\breve F) / T(\breve F) \cap \breve I.$$ The affine Weyl group is defined by $$W_a(\bG) = N(\breve F) \cap \breve G_1 / T(\breve F) \cap \breve I,$$ where $\breve G_1 \subseteq \breve G$ is the subgroup generated by all Iwahori subgroups of $\breve G$. 

Let $W_0(\bG)=N(\breve F) / T(\breve F)$ be the relative Weyl group of $\bG$. Then $$\tW(\bG) = X_*(T)_\Gamma \rtimes W_0(\bG),$$ where $\Gamma=\Gal(\overline{\breve F}/\breve F)$ is the absolute Galois group, and $X_*(T)_\Gamma$ is the group of $\Gamma$-coinvariants of the coweight lattice $X_*(T)$ of $T$.


By \cite[Appendix]{PR}, there is a reduced root system $R$ associated to $\bG$ such that $W_0(\bG) = W_0(R)$ and $W_a(\bG) = W_a(R)$. Moreover, any extended affine Weyl group can be realized as the Iwahori-Weyl group of some connected reductive group split over $F$. On the other hand, $\tW(\bG)$ for any connected reductive group $\bG$, in general, may not be an extended affine Weyl group in \S\ref{root}. The reason is that the group $X_*(T)_{\G}$ is not torsion-free in general. However, the torsion part $X_*(T)_{\G, \text{tor}}$ always lies in the center of $\tW(\bG)$ and $\tW(\bG)/X_*(T)_{\G, \text{tor}} \cong \tW(\mathfrak R)$ for some reduced root datum $\mathfrak R$ whose underlying root system is $R$. See \cite[A1]{HN18}. 

Under the identification $\tW(\bG)/X_*(T)_{\G, \text{tor}} \cong \tW(\mathfrak R)$, the length functions and Bruhat orders on $\tW(\bG)$ and $\tW(\mathfrak R)$ are compatible. We still denote the length function on $\tW(\bG)$ by $\ell$ and the Bruhat order on $\tW(\bG)$ by $\le$. 
 
The Frobenius automorphism $\s$ induces actions on $\breve G$ and on $\tW$, which are still denoted by $\s$. Then $\s$ preserves the length function $\ell$ and the Bruhat order $\le$.

\subsection{Convolution product}\label{sec:conv}
The Demazure product on $\tW(\bG)$ can be realized via the convolution product of Schubert varieties in the affine flag variety $\breve G / \breve I$. In this paper, we will present a variation of the convolution product. 

For any $w \in \tW(\bG)$, we choose a representative $\dot w$ in $N(\breve F)$. We have the decomposition $\breve G=\bigsqcup_{w \in \tW(\bG)} \breve I \dot w \breve I$. Note that $\breve I \dot w \breve I$ is an admissible subscheme of $\breve G$ in the sense of Grothendieck (see \cite[A2]{He-KR}). The closure $\overline{\breve I \dot w \breve I}$ of $\breve I \dot w \breve I$ equals $\bigsqcup_{w' \le w} \breve I \dot w' \breve I$. Let $w, w' \in \tW(\bG)$ and $w''=w \ast w'$. It is well-known that  $$\overline{\breve I \dot w \breve I} \, \cdot \, \overline{\breve I \dot w' \breve I}=\overline{\breve I \dot w'' \breve I}.$$  


\subsection{Generic Newton point}\label{sec:Newton}
For any $g, g' \in \breve G$, we set $g \cdot_\s g'=g g' \s(g) \i$. For $b \in \breve G$, we denote by $[b]$ the $\s$-conjugacy class of $b$. Let $B(\bG)$ be the set of $\s$-conjugacy classes of $\breve G$. 

By \cite[Theorem A.1]{He-KR}, each $\s$-conjugacy class is an admissible subscheme of $\breve G$. Then its closure is a union of some $\s$-conjugacy classes of $\breve G$. We denote the partial order on $B(\bG)$ (defined via the closure relation) by $\le$. One may also define other partial orders on $B(\bG)$ via two different combinatorial ways, and it is proved in \cite[Theorem 3.1]{He-KR} that these partial orders coincide. We refer to \cite[\S 3]{He-KR} for a detailed discussion on these partial orders. 

For $w \in \tW(\bG)$, we set $B(\bG)_w = \{[b] \in B(\bG); [b] \cap \breve I \dot w \breve I \neq \emptyset\}$. Since $\breve I \dot w \breve I$ is irreducible, there is a unique $\s$-conjugacy class $[b_w]$ of $\breve G$ such that the intersection $[b_w] \cap \breve I \dot w \breve I$ (as an admissible subscheme) is dense in $\breve I \dot w \breve I$. We call $[b_w]$ the {\it generic Newton point}\footnote{Here the name comes from the fact that the $\s$-conjugacy classes $[b]$ are classified by Kottwitz via the image under the Kottwitz map and the Newton point of $[b]$.} associated to $w$. Then by definition, $[b_w]$ is the unique maximal element in $B(\bG)_w$ with respect to the partial order $\le$. 

We will use the following facts on $B(\bG)$. 

\begin{enumerate}
    \item Let $w \in \tW(\bG)$. Applying Lang's theorem on $T(\breve F) \cap \breve I$,  one deduces that the $\s$-conjugacy class of $\dot w$ is independent of the choice of the representative $\dot w$ of $w$. In particular, for any $\l \in X_*(T)_{\G}$, we write $[\l]$ for the $\s$-conjugacy class of $\dot w$, where $w=t^\l$. 
    
    \item Suppose that $\bG$ is split over $F$. Let $\l, \l'$ be dominant coweights. Then $[\l]=[\l']$ if and only if $\l=\l'$. This follows from  Kottwitz's classification of $B(\bG)$ in \cite{Ko1} and \cite{Ko2} (see also \cite[Proposition 3.6]{He14}).
    
    \item Suppose that $\bG$ is quasi-split over $F$. Let $x \in W_0(\bG)$ and $\l$ be a dominant coweight. Then $[b_{x t^\l}]=[\l]$. See \cite[Theorem 4.2]{He21}. 
\end{enumerate}

Although not needed in this paper, it is worth mentioning that if $\bG$ is quasi-split over $F$ and $\l$ is a dominant coweight, then the Newton point of $[\l]$ is the $\s$-averaging of $\l$. 

The following result of Milic\'{e}vi\'{c} plays a key role in the proof of Theorem \ref{main}. 

\begin{theorem} [{\cite[Theorem 3.2]{M}}] \label{regular}
Let $\bG$ be a split group. Let $x, y \in W_0(\bG)$ and $\mu \in X_*(T)^+$. Suppose that $\mu$ is sufficiently regular. Then $$[b_{x t^\mu y}] = [\mu - \wtt(y \i, x)].$$ 
\end{theorem}

\begin{remark}
Here sufficiently regular means $\<\mu, \a\> \ge M$ for all simple root $\a$ of $\bG$, where $M$($\ge 6 \ell(w_0)$) is a large number given explicitly in \cite[(6.1)]{M}. 
\end{remark}

We also need the following result. 

\begin{lemma} \label{shift}
Let $w_1, w_2 \in \tW$. Then $[b_{w_1 \ast w_2}]=[b_{w_2 \ast \s(w_1)}]$. 
\end{lemma}

\begin{proof}
For any subset $X$ of $\breve G$, we set $\breve G \cdot_\s X=\{g \cdot_\s x; g \in \breve G, x \in X\}$. 
By definition, for any $w \in \tW$, $[b_w]$ is the unique $\s$-conjugacy class $[b]$ of $\breve G$ such that the intersection $[b] \cap \breve I \dot w \breve I$ (as an admissible subscheme) is dense in $\overline{\breve I \dot w \breve I}$. Hence $[b_w]$ is the unique maximal $\s$-conjugacy class of $\breve G$ that is contained in $\breve G \cdot_\s \overline{\breve I \dot w \breve I}$. 

Set $w_3=w_1 \ast w_2$ and $w_4=w_2 \ast \s(w_1)$. We have that 
$$\breve G \cdot_\s \overline{\breve I \dot w_3 \breve I}=\breve G \cdot_\s (\overline{\breve I \dot w_1 \breve I} \, \cdot \, \overline{\breve I \dot w_2 \breve I})=\breve G \cdot_\s (\overline{\breve I \dot w_2 \breve I} \, \cdot \, \s(\overline{\breve I \dot w_1 \breve I}))=\breve G \cdot_\s \overline{\breve I \dot w_4 \breve I}.
$$

Thus $[b_{w_3}]=[b_{w_4}]$. The statement is proved. 
\end{proof}

\smallskip

Now we prove Theorem \ref{main}. 

\subsection{Proof of Theorem \ref{main}}
Let $\bG$ be the connected, split reductive group over $F$ with root datum $\mathfrak R$. In this case, the action of $\s$ on $\tW$ is the identity map. Let $w_1=t^{\mu_1} y$ and $w_2=x t^{\mu_2}$. By our assumption, $w_2 \in \tW^{\BS_0}_{\max}$. Thus $w_2=w_2 \ast w_0$. So $$w_1 \ast w_2=w_1 \ast (w_2 \ast w_0)=(w_1 \ast w_2) \ast w_0.$$ Hence $w_1 \ast w_2 \in \tW^{\BS_0}_{\max}$. In particular, $w_1 \ast w_2=z t^\eta$ for some $z \in W_0$ and $\eta \in X_*^+$. 

Let $\chi$ be a sufficiently regular dominant coweight in the sense of \cite{M}. Set $w'_2=w_2 t^\chi=x t^{\mu_2+\chi}$. Then $w'_2=w_2 \ast t^\chi$. We have $w_1 \ast w'_2=(w_1 \ast w_2) \ast t^\chi=z t^{\eta+\chi}$. By \S\ref{sec:Newton} (3), $[b_{z t^{\eta+\chi}}]=[\eta+\chi]$.

On the other hand, let $w = w'_2 w_1 = x t^{\chi+\mu_1 + \mu_2} y$. Since $\chi + \mu_2$ is dominant and $t^{\mu_1}y \in {}^{\BS_0} \tW$, $t^{\chi+\mu_1 + \mu_2} y \in {}^{\BS_0} \tW$. Therefore, $\ell(w)=\ell(x)+\ell(t^{\mu_1+\mu_2})-\ell(y) = \ell(w_1) + \ell(w_2')$. In particular, $w=w'_2 \ast w_1$. By Theorem \ref{regular} and Lemma \ref{shift},  \[ [\chi+\mu_1+\mu_2-\wtt(y\i, x)] = [b_{x t^{\chi+\mu_1+\mu_2} y}] = [b_{w'_2 \ast w_1}] = [b_{w_1 \ast w'_2}]= [\chi+\eta].\] Since $\chi$ is sufficiently regular, both $\chi+\mu_1+\mu_2-\wtt(y\i, x)$ and $\chi+\eta$ are dominant. By \S\ref{sec:Newton} (2), $\chi+\mu_1+\mu_2-\wtt(y\i, x)=\chi+\eta$. Hence $\eta=\mu_1+\mu_2-\wtt(y\i, x)$. In particular, $\mu_1+\mu_2-\wtt(y\i, x) \in X_*^+$. 

\subsection{Some comments}\label{sec:comm} 
We have the following interesting property on the quantum Bruhat graph. Namely, for any $x, y \in W_0 (\mathfrak R)$ and any simple root $\a$ of $\mathfrak R$, we have $$\<\wtt(y, x), \a\> \le \d_{y \a}+\d_{x \a}.$$ Here for a root $\b \in R$, $\d_\b=\begin{cases} 0, & \text{ if } \b \in R^+ \\ 1, & \text{otherwise} \end{cases}.$ For the root datum of adjoint type, the claim follows directly from Theorem \ref{main} (1). The statement for arbitrary root datum can be deduced easily from the adjoint root datum. 

Theorem \ref{main} is stated for extended affine Weyl groups $\tW(\mathfrak R)$, not for the Iwahori-Weyl groups $\tW(\bG)$. As we mentioned in \S\ref{sec:twbg}, the Iwahori-Weyl groups $\tW(\bG)$, in general, may not be of the form $\tW(\mathfrak R)$ for some reduced root datum $\mathfrak R$. However, using the identification $\tW(\bG)/X_*(T)_{\G, \text{tor}} \cong \tW(\mathfrak R)$ and the fact that $X_*(T)_{\G, \text{tor}}$ lies in the center of $\tW(\bG)$, the statement in Theorem \ref{main} remains valid for arbitrary Iwahori-Weyl group $\tW(\bG)$.

\section{Applications}
\subsection{Lowest two-sided cells} 
Let $C_{\text{lowest}}$ be the lowest Kazhdan-Lusztig two-sided cell of $\tW$. The explicit description is obtained by the work of Lusztig \cite{L}, Shi \cite{S} and B\'edard \cite{B}. Let $x, y \in W_0$ and $\mu \in X_*^+$ with $t^\mu y \in {}^{\BS_0} \tW$. Then $x t^\mu y \in C_{\text{lowest}}$ if and only if $\<\mu, \a\>+\d_{x \a}-\d_{y \i \a} \neq 0$ for any $\a \in \Pi$. In other words, $C_{\text{lowest}}$ consists of elements in the shrunken Weyl chambers. The latter terminology is often used in the literature on the affine Deligne-Lusztig varieties. The following result generalizes Theorem \ref{regular} to the elements in the lowest two-sided cell, which was conjectured by Milic\'{e}vi\'{c} in \cite[\S 6.3]{M} for split groups. 


\begin{proposition}\label{prop:newton}
Suppose that $\bG$ is quasi-split and adjoint over $F$. Let $x, y \in W_0(\bG)$ and $\mu \in X_*(T)_{\G}^+$ with $t^\mu y \in {}^{\BS_0} \tW(\bG)$. If $x t^\mu y \in C_{\text{lowest}}$, then $[b_{x t^\mu y}]=[\mu-\wtt(y \i, \s(x))]$. 
\end{proposition}
\begin{remark}
One important application of the explicit formula of the generic Newton points is to study the cordial elements introduced by Milic\'{e}vi\'{c} and Viehmann \cite{MV}. One may use proposition \ref{prop:newton} to extend Theorem 1.2 (b), (c) and Theorem 4.2 in \cite{MV} from the superregular elements to the elements in the lowest two-sided cells and from split groups to quasi-split groups. 
\end{remark}

\begin{proof}
Set $w=x t^\mu y$. We simply write $\tW(\bG)$ by $\tW$ and write $W_0(\bG)$ by $W_0$. By \cite[Proposition 2.1.5 \& Corollary 2.1.7]{LLHLM}, $w=w_1 w_2$ for some $w_1 \in \tW^{\BS_0}_{\max}$ and $w_2 \in {}^{\BS_0} \tW$ with $\ell(w)=\ell(w_1)+\ell(w_2)$. Since $w_1 \in \tW^{\BS_0}_{\max}$, we have $w_1=x' t^{\mu_1}$ for some $x' \in W_0$ and $\mu_1 \in X_*(T)_{\G}^+$. Since $w_2 \in {}^{\BS_0} \tW$, we have $w_2=t^{\mu_2} y'$ for some $y' \in W_0$ and $\mu_2 \in X_*(T)_{\G}^+$. Thus $w=x' t^{\mu_1+\mu_2} y'$. It is easy to see that $t^{\mu_1+\mu_2} y \in {}^{\BS_0} \tW$. So we have $x'=x, y'=y$ and $\mu=\mu_1+\mu_2$. 

By Theorem \ref{main} and the comments in \S\ref{sec:comm}, $\mu_1 + \s(\mu_2) - \wtt(y \i, \s(x))$ is dominant and $w_2 \ast \s(w_1)=z t^{\mu_1 + \s(\mu_2) - \wtt(y \i, \s(x))}$ for some $z \in W_0$. By Lemma \ref{shift} and \ref{sec:Newton} (3), $$[b_w]=[b_{w_1 \ast w_2}]=[b_{w_2 \ast \s(w_1)}]=[b_{z t^{\mu_1 + \s(\mu_2) - \wtt(y \i, \s(x))}}]=[\mu_1 + \s(\mu_2) - \wtt(y \i, \s(x))].$$ 

Moreover, $\mu_1 + \s(\mu_2) - \wtt(y \i, \s(x))+(\mu_2-\s(\mu_2))=\mu- \wtt(y \i, \s(x))$. Thus $[\mu_1 + \s(\mu_2) - \wtt(y \i, \s(x))]=[\mu- \wtt(y \i, \s(x))]$. The statement is proved.
\end{proof}

\subsection{Generic Newton points: reduction to quasi-split groups}\label{sec:3.2}
The study of generic Newton points for arbitrary reductive groups can be reduced to the quasi-split adjoint groups. 

Let $\bG_{\text{ad}}$ be the adjoint group of $\bG$. Let $\s_{\ad}$ be the Frobenius automorphism of $\bG_{\ad}$. Let $w \in \tW(\bG)$ and $w_{\ad}$ be its image in $\tW(\bG_{\ad})$. The maximal element $[b_w]^{\bG}$ of $B(\bG)_w$ is determined by its image $[b_{w_{\ad}}]^{\bG_{\ad}} \in B(\bG_{\ad})_{w_{\ad}}$ and its image under the Kottwitz map $\k: B(\bG) \to \pi(\bG)_\s$. 

We then study $[b_{w_{\ad}}]^{\bG_{\ad}}$. Note that there exists a length-zero element $\t$ in $\tW(\bG_{\ad})$ such that $\Ad(\t) \circ \s$ preserves the set of simple reflections $\BS_0 \subset \tilde \BS$. Set $\s_0=\Ad(\dot \t) \circ \s$. Let $\mathbf H$ be the associated inner form of $\bG_{\ad}$. Since $\s_0(\BS_0)=\BS_0$, $\mathbf H$ is a quasi-split inner form of $\bG_{\ad}$. We have $\bG_{\ad}(\breve F)=\mathbf H(\breve F)$. It is easy to see that the map $b \mapsto b \dot \t$ induces a natural bijection $B(\bG_{\ad}) \cong B(\mathbf H)$ and this bijection preserves the partial order $\le$ defined via the closure relations in $\bG_{\ad}(\breve F)=\mathbf H(\breve F)$. Under the map $b \mapsto b \dot \t$, $\breve I \dot w_{\ad} \breve I$ is mapped to $\breve I \dot w_{\ad} \dot \t \breve I$. Hence the natural bijection $B(\bG_{\ad}) \cong B(\mathbf H)$ restricts to a natural bijection $B(\bG_{\ad})_{w_{\ad}} \cong B(\mathbf H)_{w_{\ad} \t}$ for any $w_{\ad} \in \tW(\bG_{\ad})=\tW(\mathbf H)$. In particular, the maximal element $[b_{w_{\ad}}]^{\bG_{\ad}}$ of $B(\bG_{\ad})_{w_{\ad}}$ corresponds to the maximal element $[b_{w_{\ad} \t}]^{\mathbf H}$ of $B(\mathbf H)_{w_{\ad} \t}$. 

\subsection{Demazure product in the lowest two-sided cell}

If $w \in \tW_{\text{lowest}}$ and $s \in \tilde \BS$ such that $w \le w s$, then it follows by definition that $w s \in \tW_{\text{lowest}}$. Similarly, if $w \le s w$, then $s w \in \tW_{\text{lowest}}$. Thus $w \ast w' \in \tW_{\text{lowest}}$ if one of $w, w'$ belongs to $\tW_{\text{lowest}}$. This observation is pointed out to us by Felix Schremmer.

Now we give an explicit formula for the Demazure product of elements in $C_{\text{lowest}}$. 

\begin{proposition}\label{prop:ww}
Let $x, x', y, y' \in W_0$ and $\mu, \mu' \in X_*^+$ with $t^{\mu} y, t^{\mu'} y' \in {}^{\BS_0} \tW$. If $x t^\mu y, x' t^{\mu'} y' \in C_{\text{lowest}}$, then $$(x t^{\mu} y) \ast (x' t^{\mu'} y')=x t^{\mu+\mu'-\wtt(y \i, x')} y'.$$ 
\end{proposition}

\begin{proof}
Let $\tW=\tW(\mathfrak R)$. We first consider the case where $\mathfrak R$ is a root datum of adjoint type. 

As in the proof of Proposition \ref{prop:newton}, we have $x t^{\mu} y=w_1 \ast w_2$ and $x' t^{\mu'} y'=w'_1 \ast w'_2$, where $w_1=x t^{\mu_1} \in \tW^{\BS_0}_{\max}$, $w_2=t^{\mu_2} y \in {}^{\BS_0} \tW$, $w'_1=x' t^{\mu'_1} \in \tW^{\BS_0}_{\max}$, $w'_2=t^{\mu'_2} y' \in {}^{\BS_0} \tW$ for some dominant coweights $\mu_1, \mu_2, \mu'_1, \mu'_2$ with $\mu=\mu_1+\mu_2$ and $\mu'=\mu'_1+\mu'_2$. 

By Theorem \ref{main}, $w_2 \ast w'_1=z t^{\mu'_1+\mu'_2-\wtt(y \i, x')}$ for some $z \in W_0$ and $\mu'_1+\mu'_2-\wtt(y \i, x')$ is dominant. Hence \begin{align*}
    (x t^{\mu} y) \ast (x' t^{\mu'} y') &=(w_1 \ast w_2) \ast (w'_1 \ast w'_2)=w_1 \ast (w_2 \ast w'_1) \ast w'_2 \\ &=w_1 \ast (z t^{\mu'_1+\mu_2-\wtt(y \i, x')}) \ast w'_2 \\ &=(w_1 \ast z) \ast (t^{\mu'_1+\mu_2-\wtt(y \i, x')} \ast w'_2) \\ &=w_1 \ast (t^{\mu'+\mu_2-\wtt(y \i, x')} y') \\ &=x t^{\mu+\mu'-\wtt(y \i, x')} y'.
\end{align*} Notice that $w_1 \in \tW^{\BS_0}_{\max}$ and $t^{\mu'+\mu_2-\wtt(y \i, x')} y' \in {}^{\BS_0} \tW$. 

Now we consider the general case. Let $\mathfrak R_{\ad}$ be the root system of adjoint type associated to $\mathfrak R$ and $\tW_{\ad}=\tW(\mathfrak R_{\ad})$. Then we have a natural projection map $\pi_{\ad}: \tW \to \tW_{\ad}$. For any $\l \in X_*$, we denote by $\l_{\ad}$ its image in $(X_*)_{\ad}$. Then we have $$\pi_{\ad}((x t^{\mu} y) \ast (x' t^{\mu'} y'))=(x t^{\mu_{\ad}} y) \ast (x' t^{\mu'_{\ad}} y')=x t^{\mu_{\ad}+\mu'_{\ad}-\wtt(y \i, x')} y'.$$ On the other hand, $(x t^{\mu} y) \ast (x' t^{\mu'} y') \in W_a t^{\mu+\mu'} W_a$. So $(x t^{\mu} y) \ast (x' t^{\mu'} y')=x t^{\mu+\mu'-\wtt(y \i, x')} y'$. 
\end{proof}

\subsection{Demazure product on the coweight lattice}\label{sec:3.4} For any coweight $\l \in X_*$, we denote by $\bar \l$ the unique dominant coweight in the $W_0$-orbit of $\l$. We define the map $pr: \tW \to X_*$, which sends any element $w \in \tW$ to the unique dominant coweight $\l$ with $w \in W_0 t^\l W_0$. 

Now we define the Demazure product on the coweight lattice by $$\ast: X_* \times X_* \to X_*^+, \quad (\l_1, \l_2) \mapsto pr(t^{\l_1} \ast t^{\l_2}).$$

We have the following explicit formula. 

\begin{proposition}\label{prop:ast-l}
Let $\l_1, \l_2 \in X_*$. For $i=1, 2$, let $w_i$ be the unique element in $W_0^{I(\bar \l_i)}$ with $\l_i=w_i(\bar \l_i)$. Then $$\l_1 \ast \l_2=\bar \l_1+\bar \l_2-\wtt(w_1, w_2 w_{I(\bar \l_2)}).$$
\end{proposition}

\begin{proof}
We have $t^{\l_1}=w_1 t^{\bar \l_1} w_1 \i$. Since $w_1 \in W_0^{I(\bar \l_1)}$, $t^{\bar \l_1} w_1 \i \in {}^{\BS_0} \tW$. We have $t^{\l_2}=w_2 t^{\bar \l_2} w_2 \i \in w_2 w_{I(\bar \l_2)} t^{\bar \l_2} W_0$. Since $w_2 \in W_0^{I(\bar \l_2)}$, $w_2 w_{I(\bar \l_2)} t^{\bar \l_2} \in \tW^{\BS_0}_{\max}$. By Theorem \ref{main}, $\bar \l_1+\bar \l_2-\wtt(w_1, w_2 w_{I(\bar \l_2)}) \in X_*^+$ and 
\begin{align*}
    t^{\l_1} \ast t^{\l_2} & \in (w_1 t^{\bar \l_1} w_1 \i) \ast (w_2 w_{I(\bar \l_2)} t^{\bar \l_2} W_0) \\ &=\bigl((w_1 t^{\bar \l_1} w_1 \i) \ast (w_2 w_{I(\bar \l_2)} t^{\bar \l_2}) \bigr) W_0 \\ &=w_1 \ast \bigl( (t^{\bar \l_1} w_1 \i) \ast (w_2 w_{I(\bar \l_2)} t^{\bar \l_2}) \bigr) W_0 \\ & \subseteq w_1 \ast (W_0 t^{\bar \l_1+\bar \l_2-\wtt(w_1, w_2 w_{I(\bar \l_2)})}) W_0 \\ &= W_0 t^{\bar \l_1+\bar \l_2-\wtt(w_1, w_2 w_{I(\bar \l_2)})}) W_0.
\end{align*}

The statement is proved. 
\end{proof}

\subsection{Lusztig-Vogan map}
Define $\iota: \tW \to \tW$ by $\iota(t^\l y)=w_0 t^{-\l} y w_0$ for $\l \in X_*$ and $y \in W_0$. It is easy to see that $\iota$ is a length-preserving involutive group automorphism on $\tW$. Let $\mathbf{I}_\iota = \{w \in \tW; \iota(w) = w\i\}$ be the set of twisted involutions of $\tW$. For $x \in \mathbf{I}_\iota$ we define $$ \pi_x: \tW \to \mathbf{I}_\iota, \quad w \mapsto w \ast x \ast \iota(w)\i.$$  

In \cite{LV12} and \cite{Lu12}, Lusztig and Vogan constructed a module $M$ of the Hecke algebra $H$ of $\tW$ over $\BZ[q]$, which has a linear basis indexed by $\mathbf{I}_\iota$. If we take $q = 0$, then $H$ becomes the $0$-Hecke algebra $H_0$ in \S\ref{product} and $M$ becomes a module $M_0 = \oplus_{x \in \mathbf{I}_\iota} \BZ a_x$ of the 0-Hecke algebra $H_0$. It is proved in \cite{LV} that the action of $H_0$ on $M_0$ has the following simple expression $$t_w a_x = (-1)^{\ell(x) + \ell(w) + \ell(\pi_x(w))} a_{\pi_x(w)}.$$

A particular interesting case is $x=1$ and $w \in {}^{\BS_0} \tW_{\max}$. Note that the map $\l \mapsto w_0 \ast t^\l$ gives a bijection from $X_*$ to ${}^{\BS_0} \tW_{\max}$. Thus the map $\pi_1$ induces $${}^{\BS_0} \pi: X_* \to X_*^+, \qquad \l \mapsto pr\bigl((w_0 \ast t^{\l}) \ast \iota(w_0 \ast t^{\l}) \i\bigr).$$

As a consequence of proposition \ref{prop:ast-l}, we obtain the following formula for ${}^{\BS_0} \pi$. This answers a question of Lusztig and Vogan in \cite[0.4]{LV}.

\begin{corollary}\label{cor:LV}
Let $\l \in X_*$ and $w$ be the unique element in $W_0^{I(\bar \l)}$ with $\l=w(\bar \l)$. Then ${}^{\BS_0} \pi(\l)=2 \bar \l-\wtt(w, w_0 w)$. 
\end{corollary}

\begin{proof}
Note that $(w_0 \ast t^{\l}) \ast \iota(w_0 \ast t^{\l}) \i=w_0 \ast (t^{\l} \ast \iota(t^{\l}) \i) \ast w_0 \in W_0 (t^{\l} \ast \iota(t^{\l}) \i) W_0$. Moreover, $\iota(t^{\l}) \i=t^{w_0(\l)}$. Thus ${}^{\BS_0} \pi(\l)=\l \ast (w_0(\l))$. Note that $\overline{w_0(\l)}=\bar \l$ and $w_0(\l)=w_0 w(\bar \l)$, where $w_0 w \in (W_0^{I(\bar \l)})_{\max}$. By Proposition \ref{prop:ast-l}, $\l \ast (w_0(\l))=2 \bar \l-\wtt(w, w_0 w)$. 
\end{proof}

\end{document}